\documentclass[reqno, 11pt,fleqn]{amsart}
\usepackage{graphicx} 
\usepackage[margin= 2.5cm,bottom=19mm]{geometry}
\usepackage{amsthm}
\usepackage{thm-restate}
\usepackage{geometry}
\usepackage{amssymb}
\usepackage{amsmath,mathtools}
\usepackage{hyperref}
\usepackage[capitalize]{cleveref}
\usepackage{comment}
\usepackage{xcolor}
\usepackage{enumerate}
\usepackage{enumitem}
\usepackage{listings}
\usepackage{complexity}
\usepackage{blkarray}
\usepackage{lmodern}
\usepackage[english]{babel}
\usepackage{accents}
\usepackage{float}
\usepackage{tikz}
\usetikzlibrary{decorations.pathmorphing}
\usepackage{bbm}
\usepackage[normalem]{ulem}
\newcommand{\stkout}[1]{\ifmmode\text{\sout{\ensuremath{#1}}}\else\sout{#1}\fi}

\newtheorem{theorem}{Theorem}[section]
\newtheorem{lemma}[theorem]{Lemma}

\newtheorem{conjecture}[theorem]{Conjecture}

\DeclarePairedDelimiter{\parens}{(}{)}
\DeclarePairedDelimiter{\set}{\{}{\}}

\DeclarePairedDelimiter{\size}{|}{|}

\theoremstyle{definition}
\newtheorem{definition}[theorem]{Definition}
\newtheorem{observation}[theorem]{Observation}
\crefname{observation}{Observation}{Observations}

\newtheorem{problem}[theorem]{Problem}

\title[Almost-perfect colorful matchings in three-edge-colored bipartite graphs]{Almost-perfect colorful matchings in three--edge-colored bipartite graphs}
\author[S.Boyadzhiyska]{Simona Boyadzhiyska}
\address[SB]{HUN-REN Alfréd Rényi Institute of Mathematics, Budapest, Hungary.}
\email{simona@renyi.hu}

\author[M. Christoph]{Micha Christoph}
\address[MC]{Department of Mathematics, ETH Z\"urich}
\email{micha.christoph@math.ethz.ch}

\author[T. Szab\'o]{Tibor Szab\'o}
\address[TSz]{Institut f\"ur Mathematik, Freie Universit\"at Berlin, Berlin, Germany}
\email{szabo@mi.fu-berlin.de}

\date{\today}

\newcommand{\cal}{\mathcal}

\renewcommand{\tilde}{\widetilde}

\begin{document}

\begin{abstract}
    We prove that, for positive integers $n,a_1, a_2, a_3$ satisfying $a_1+a_2+a_3 = n-1$, it holds that any bipartite graph $G$ which is the union of three perfect matchings $M_1$, $M_2$, and $M_3$ on $2n$ vertices contains a matching $M$ such that $|M\cap M_i| =a_i$ for $i= 1,2,$ and $3$. The bound~$n-1$ on the sum is best possible in general. Our result verifies the multiplicity extension of the Ryser-Brualdi-Stein Conjecture, proposed recently by Anastos, Fabian, M\"uyesser, and Szab\'o, for three colors.
\end{abstract}

\maketitle

\section{Introduction}

An edge-colored graph is \emph{rainbow} if all of its edges have distinct colors. In 1967 Ryser~\cite{ryser1967neuere} (see also~\cite{best2018did}) conjectured that for odd $n$ it holds that in any proper $n$-coloring of the complete bipartite graph $K_{n,n}$ there is a rainbow matching of size $n$. This conjecture proved to be extremely influential in Combinatorics, motivating the development of newer and more sophisticated approaches employing a widening variety of tools. It is still open today, though significant progress has been made over the years.  

It is not hard to see that the statement of the conjecture does not hold for even $n$, though it has been subsequently conjectured that it only barely does not. More precisely, the Ryser-Brualdi-Stein Conjecture~\cite{brualdi1991combinatorial,ryser1967neuere,stein1975transversals} states that any proper edge-coloring of $K_{n,n}$ contains a rainbow matching of size~$n-1$.  The proof of this conjecture was announced recently by Montgomery~\cite{montgomery2023proof}, following earlier progress by Koksma~\cite{koksma1969}, Drake~\cite{drake1977}, Brouwer, de Vries, and Wieringa~\cite{brouweretal1978}, Woolbright~\cite{woolbright1978}, Hatami and Shor~\cite{hatamishor2008,shor1982}, and most recently by Keevash, Pokrovskiy, Sudakov, and Yepremyan~\cite{keevash2022ryser}. This takes us deceptively close to the full resolution of the conjecture of Ryser for odd $n$, though, according to~\cite{montgomery2023proof}, significant novel ideas will likely be necessary to find that last edge completing a rainbow perfect matching. 

Motivated by earlier work of Arman, R\"odl, and Sales~\cite{arman2023colorful}, Anastos, Fabian, M\"uyesser, and Szab\'o~\cite{anastos2023splitting} 
investigated what other matching structures, in addition to rainbow matchings, would be inevitable in a proper edge-coloring of $K_{n,n}$.
They conjectured an extension of Montgomery's Theorem, in which one seeks a matching of size $n-1$ containing a preset amount of edges from each color class. This was also raised as a question independently by Alon~\cite{Alon-question}.

\begin{conjecture}[Multiplicity Ryser-Brualdi-Stein Conjecture \cite{anastos2023splitting}]\label{con:AFMS}
Let $G$ be a complete bipartite graph on $2n$ vertices whose edge set is decomposed into perfect matchings $M_i$, $i=1,\ldots, n$. Let $a_1, \ldots , a_n\in \mathbb{N}_0$ be nonnegative integers such that $\sum_i a_i=n-1$. Then there exists a matching $M$ in $G$ such that $|M\cap M_i| = a_i$ for every $i\in\{1,\ldots, n\}$.
\end{conjecture}

By setting $a_i=1$ for all $i\in \{1,\ldots, n-1\}$ and $a_n=0$ in \cref{con:AFMS}, we obtain a  strengthening of the Ryser-Brualdi-Stein Conjecture.
If there is exactly one nonzero color-multiplicity $a_i=n-1$, then one can just take a subset of the corresponding perfect matching $M_i$. It is also easy to show that \cref{con:AFMS} holds when there are two nonzero color-multiplicities, that is, $a_i+a_j=n-1$. Indeed, the union of the perfect matchings $M_i$ and $M_j$ forms a $2$-factor consisting of even cycles, which alternate between the two matchings. To create $M$, we simply pick  edges from $M_i$ component by component, until in some component $C$ we reach the target number $a_i$ of edges designated for $M_i$. 
In $C$, we pick the necessary amount of $M_i$-edges in a consecutive fashion. Then, after potentially leaving unsaturated the two vertices incident to the first and last selected $M_i$-edges in $C$, we can saturate the remaining $2n-2a_i-2=2a_j$ vertices with $M_j$-edges. 

For three nonzero color-multiplicities $a_i, a_j, a_k$, Anastos, Fabian, M\"uyesser, and Szab\'o~\cite{anastos2023splitting} showed that a matching with the desired color-multiplicities exists provided that $a_i+a_j+a_k \leq n-2$. In fact, they showed the result in the more general setting where the three matchings need not form a bipartite graph (and also need not be disjoint); they also proved that in the general case the bound~$n-2$ on the sum $a_i+a_j+a_k$ is best possible (see~\cite[Remark 2]{anastos2023splitting}).

In the present paper we extend this to a full proof of the Multiplicity Ryser-Brualdi-Stein Conjecture for three nonzero color-multiplicities. The proof is elementary, using only augmenting paths, but turns out to be surprisingly delicate. 

\subsection{Main result}
To state our result we first introduce some convenient terminology.
Let $G$ be a graph which is the union of three disjoint perfect matchings $M_1$, $M_2$, and $M_3$.
For integers $a_1,a_2,a_3 \in \mathbb{N}_0$, 
a matching $M$ of $G$ is called an \emph{$(a_1,a_2,a_3)$-matching} if $|M\cap M_i| =a_i$ for every $i\in \{1,2,3\}$. 
Our main result states that in a bipartite graph a matching of size $n-1$ can be found for any color-multiplicity triple summing up to $n-1$. 

\begin{theorem}\label{thm:main}
    Let $G$ be a bipartite graph on $2n$ vertices which is the union of three disjoint perfect matchings $M_1$, $M_2$, and $M_3$.
    Then, for any integers $a_1,a_2,a_3\in \mathbb{N}_0$ satisfying  $a_1 + a_2 + a_3 = n-1$, the graph $G$ contains an $(a_1,a_2,a_3)$-matching.
\end{theorem}

As shown in~\cite[Proposition~1]{anastos2023splitting}, \cref{thm:main} is best possible: for any $n$ and $0\leq a_1,a_2,a_3\leq n-1$ summing up to $n$, there exists a bipartite graph on $2n$  vertices which is the disjoint union of three perfect matchings and has  no $(a_1,a_2,a_3)$-matching.
Note also that, by K\"onig's Theorem, any collection  of $k$ pairwise disjoint perfect matchings of $K_{n,n}$ can be extended to a collection of $n$ pairwise disjoint perfect matchings.
Therefore, the existence of an $(a_1,\ldots,a_n)$-matching in the Multiplicity Ryser-Brualdi-Stein Conjecture with three nonzero color-multiplicities $a_i, a_j, a_k$ summing up to $n-1$ is equivalent 
to the existence of an $(a_i,a_j, a_k)$-matching in \cref{thm:main}. 
As discussed earlier, the analog of \cref{thm:main} is false if we remove the assumption that $G$ is bipartite; however, it is possible that the conclusion of the theorem holds under a weaker assumption about the structure of $G$. We discuss this in more detail in the concluding remarks.

\subsection{Terminology and notation} 
We say that a vertex $v$ is \emph{saturated} by a matching $M$ if $M$ contains an edge incident to $v$; otherwise $v$ is \emph{unsaturated}; if we want to explicitly specify the matching $M$ we sometimes call a vertex $M$-saturated or $M$-unsaturated. If $v$ is saturated by  $M$, we write $M(v)$ for the \emph{matching partner} of $v$ in $M$, that is, the unique vertex $w$ such that $vw\in M$. 
Let $G$ be a graph which is the union of three disjoint perfect matchings $M_1$, $M_2$ and $M_3$.  For a matching~$M$ in~$G$, we write  $a_i(M) = \size{M\cap M_i}$ for all $i\in [3]$. 
Let $P = (v_1,\dots,v_\ell)$ be a path in~$G$. For $1\leq i\leq j\leq\ell$, we write $P_{i,j}$ for the sub-path from $v_i$ to $v_j$, that is, $P_{i,j}=(v_i,v_{i+1},\dots,v_j)$. For a matching $M$, the path $P$ is \emph{$M$-alternating} if $P$ alternates between edges in $M$ and edges not in~$M$. Given two subgraphs $F$ and $F'$, the path $P$ is \emph{$(F,F')$-alternating} if $P$ alternates between edges of~$F$ and edges of~$F'$. We define $M$- and $(F,F')$-alternating cycles similarly. Given paths $P$ and $P'$ such that the last vertex of $P$ is the first vertex of $P'$, we write $P\star P'$ for the path obtained by concatenating the two paths.

\subsection{Organization of the paper} In the next section, we introduce the two main ingredients of our proof, which we call the Reduction Lemma and the Switching Lemma, and deduce Theorem~\ref{thm:main} from them. Then, we prove the Reduction Lemma in Section~\ref{sec:reduction} and the Switching Lemma in Section~\ref{sec:switch}. Finally, we propose several different avenues to extend our result in the last section.

\section{Proof of main result}\label{sec:main-proof}

The goal of this section is to prove \cref{thm:main}. The proof consists of two main ingredients. The first is a reduction to the connected case.
While in many theorems of graph theory it is simple to extend the statement to all graphs once it is known for all \emph{connected} graphs, in our setting this step requires a separate and nontrivial argument. The issue is that looking for a matching of size $n-1$ in a disconnected graph $G$ means that we need a perfect matching in all but one component. The next lemma guarantees that we can find such perfect matchings satisfying the given requirements. We will prove \cref{lem:reduction} in \cref{sec:reduction}. 

\begin{lemma}[Reduction Lemma]\label{lem:reduction}
    Let $F$ be a bipartite graph on $2m$ vertices which is the union of three disjoint perfect matchings $M_1,M_2$, and $M_3$. Let ${b_1,b_2,b_3\in \mathbb{N}_0}$ be integers such that $b_1+b_2+b_3\geq 2m-1$. Then, there exists a perfect matching in $F$ with at most $b_i$ edges of $M_i$ for each $i\in\{1,2,3\}$.
\end{lemma}

The second main ingredient is a ``color-switching'' lemma for connected graphs, showing that, for any matching of size $n-1$, we can find another matching of size $n-1$ with one more edge of color~$i$ and one fewer edge of color $j$ (provided there is at least one edge of color $j$). 

\begin{lemma}[Switching Lemma]\label{lem:main}
    Let $G$ be a connected bipartite graph on $2n$ vertices which is the union of three disjoint perfect matchings $M_1$, $M_2$, and $M_3$. Let $a_1,a_2,a_3\in \mathbb{N}_0$ be integers satisfying $a_3\geq 1$ and $a_1 + a_2 + a_3 = n-1$ and suppose that $G$ contains an $(a_1,a_2,a_3)$-matching.
    Then $G$ also contains an $(a_1+1,a_2,a_3-1)$-matching and an $(a_1, a_2+1, a_3-1)$-matching.
\end{lemma}

The proof of \cref{lem:main} is where the bulk of the work lies and will be given in \cref{sec:switch}.
We now use \cref{lem:reduction,lem:main} to deduce \cref{thm:main}.

\begin{proof}[Proof of \cref{thm:main}]
We prove the statement by induction on the number of connected components of $G$.
For the base case assume 
that $G$ is connected.
We can then apply \cref{lem:main} successively $a_1$ times in the first coordinate to a $(0,0,n-1)$-matching of $G$ (any $n-1$ edges of $M_3$)
and obtain an $(a_1,0,a_2+a_3)$-matching.  
Then, applying \cref{lem:main} $a_2$ times in the second coordinate,
we find the desired $(a_1,a_2,a_3)$-matching of $G$.

For the induction step, let $G$ be disconnected. Then there exists a component $F$ of $G$ with at most $n$ vertices. Note that the matchings $M_1, M_2$, and $M_3$ are perfect matchings in $F$ as well. Thus, $F$ is a balanced bipartite graph on $2m$ vertices for some  $m\leq \frac{n}{2}$. Now, setting $b_i = a_i$ for all $i\in \{1,2,3\}$, we have $b_1+b_2+b_3 = n-1\geq 2m-1$. Hence \cref{lem:reduction} implies that there is a perfect matching $M_F$ in $F$ using  $b'_i\leq b_i$ edges of $M_i$ for each $i\in [3]$. Note that $b_1'+b_2'+b_3' = m$.

Now consider the graph $G' = G-V(F)$. This is a bipartite graph on $2(n-m)$ vertices that is the union of three disjoint perfect matchings. We have $b_i-b_i'\geq 0$ for each $i\in \{1,2,3\}$ and $(b_1-b_1')+(b_2-b_2')+(b_3-b_3') = n-m-1$. Since $G'$ has fewer connected components than~$G$, it follows by the induction hypothesis that $G'$ contains a  $(b_1-b_1',b_2-b_2',b_3-b_3')$-matching $M_{G'}$. Then, $M_F\cup M_{G'}$ is an $(a_1,a_2,a_3)$-matching in $G$, as desired. 
\end{proof}

\section{Reduction Lemma}\label{sec:reduction}

In this section, we prove \cref{lem:reduction} by proving the following stronger result. 
    \begin{lemma}\label{lem:reduction-strengthening}
        Let $F$ be a bipartite graph on $2m$ vertices  which is the union of three disjoint perfect matchings $M_1,M_2$, and $M_3$. Let $b_1,b_2,b_3\in \mathbb{N}_0$ be integers such that $b_1+b_2+b_3\geq 2m-1$ and $b_1+b_2\geq m$. Let $M\subseteq M_1\cup M_2$ be a matching such that $(M_1\cup M_2)-V(M)$ is a disjoint union of cycles, each of length at least two, and $|M\cap M_i|\leq b_i$ for $i\in\{1,2\}$. Then, $M$ can be extended to a perfect matching in $F$ with at most $b_i$ edges of $M_i$ for $i\in\{1,2,3\}$.
    \end{lemma}

Before proving \cref{lem:reduction-strengthening}, let us briefly explain how \cref{lem:reduction} follows. As $b_1+b_2+b_3\geq 2m-1$, we may assume without loss of generality that $b_1+b_2\geq m$. Since the union of two disjoint perfect matchings is a disjoint union of cycles, taking $M$ to be the empty matching in \cref{lem:reduction-strengthening} immediately gives \cref{lem:reduction}.

    \begin{proof}[Proof of \cref{lem:reduction-strengthening}]
        We proceed by induction on the number of cycles in $(M_1\cup M_2)-V(M)$. The majority of the work is in the base case, when there is a single cycle (if there are no cycles, then~$M$ is already a perfect matching). We first introduce some notation.
        Let $b_1' = b_1-|M\cap M_1|$ and $b_2' = b_2-|M\cap M_2|$ and, without loss of generality, assume $b'_1\geq b'_2$. Let $m'= m-|M|$ and note that there are $2m'$ unsaturated vertices. We write $C$ for a shortest cycle in $(M_1\cup M_2) - V(M)$ and denote its length by $2\ell$.
        Note that 
        \begin{align}\label{eq:b_1'b_2'}
            2b_1' \geq b_1'+b_2'=b_1-|M\cap M_1|+b_2-|M\cap M_2| = b_1+b_2-|M|\geq m - (m-m')=m'.
        \end{align}

        First, suppose that $(M_1\cup M_2)-V(M)$ contains a single cycle, that is, $\ell = m'$. If $b_1' \geq m'$, we can extend $M$ by simply adding the $m'$ edges in $C\cap M_1$ to $M$ to obtain the required matching. Assume then that $b_1' < m'$. Write $C = (v_1, \dots, v_{2m'}, v_1)$. Without loss of generality, assume that $v_1v_2\in M_1$ and $v_{2m'}v_1\in M_2$.
        Now let $M'$ be the matching obtained from $M$ by adding the edges $v_{2i}v_{2i+1}$ for $i\in [m'-b_1'-1]$, that is, the first $m'-b_1'-1$ edges of $M_2$ appearing on $C$, and the edges
        $v_{2(m'-j)-1}v_{2(m'-j)}$ for $0\leq j\leq m'-b_2'-2$, that is, the last $m'-b_2'-1$ edges of $M_1$ appearing on~$C$. 
        Note that $2(m'-b_1'-1)+1 < 2(m'- (m' - b_2'-2))-1$, since $m'-2 < b_1'+ b_2'$ by~\eqref{eq:b_1'b_2'}, so~$M'$ is indeed a matching. 

        Consider now a longest $(M_3, M')$-alternating path $P$ starting from $v_1$. Note that $v_1$ is unsaturated in $M'$, so the first edge of $P$ is in $M_3$; since $M_3$ is a perfect matching, the last edge of $P$ must also belong to $M_3$, and therefore in particular the length of $P$ is odd. Suppose the other endpoint of $P$ is~$v_k$ and note that $v_k$ is also not saturated by $M'$. Therefore, $k\in \set{2(m'-b_1'-1)+2, \dots, 2(b_2'+2)-2}$. Further, since the length of $P$ is odd and $G$ is bipartite, $k$ must be even. Hence, $v_{k-1}v_k\in M_1$ and $v_kv_{k+1}\in M_2$. We also have
        \begin{align}
            |P\cap M_3|\leq|M'|+1  
            &= 2m'- (b_1'+b_2') - 2 +|M|+1\notag \\
            &= 2m'+2|M|-1-b_1-b_2 =2m-1-b_1-b_2 \leq b_3,\label{eq:P cap M_3}
        \end{align}
        where for the second equality we used the definition of $b_1'$ and $b_2'$ and for the last equality the definition of $m'$.

        Define $M''\supseteq M'$ to be the matching obtained from $M$ by adding all $M_2$-edges on the path $(v_1,\dots, v_k)$ and all $M_1$-edges on the path $(v_k,\dots, v_{2m'}, v_1)$. Then, $M''$ saturates all vertices of $C$ except for $v_1$ and $v_k$. Since $k\leq 2(b_2'+1)$, we have $\size{M'' \cap M_2}\leq \size{M\cap M_2} + b_2' \leq b_2$, and similarly, the fact that $k\geq 2(m'-b_1')$ implies that $\size{M'' \cap M_1}\leq \size{M\cap M_1} + b_1' \leq b_1$. However, $M''$ is not a perfect matching, so we need one final modification to obtain the required perfect matching.
        
        To this end, define $M^\ast = M''\Delta P$, and note that $M^\ast$ is a perfect matching of $F$ satisfying $M^\ast \cap M_1 \subseteq M''\cap M_1$ and $M^\ast \cap M_2 \subseteq M''\cap M_2$; this is because $P$ is an $(M_3, M')$-alternating path and $M'\subseteq M''$. Additionally, we have $\size{M^\ast \cap M_3} = \size{P\cap M_3}\leq b_3$ by~\eqref{eq:P cap M_3}. This completes the proof of the base case.
        \medskip
        
        Now suppose that $(M_1\cup M_2)-V(M)$ contains more than one cycle. Since the number of vertices left unsaturated by $M$ is $2m'$ and  $C$ is a cycle of shortest length, it follows that $2\ell\leq m'\leq 2b_1'$, where the final step uses~\eqref{eq:b_1'b_2'}; therefore $\ell \leq b_1'$. 
        Hence, we may add the edges in $C\cap M_1$ to $M$ without exceeding the permitted number of $M_1$-edges. Write  $M' = M\cup (C\cap M_1)$ and note that $(M_1\cup M_2)-V(M')$ has one fewer cycle than $(M_1\cup M_2)-V(M)$ and $|M'\cap M_1|= |M\cap M_1|+\ell\leq b_1$. Thus, the claim follows by induction.
    \end{proof}

\section{Switching Lemma}\label{sec:switch}

It remains to prove \cref{lem:main}. We begin by presenting some auxiliary results that allow us to gather some information about a connected graph $G$ of the given form. To simplify the presentation, throughout this section, we will demonstrate how to obtain an $(a_1+1,a_2,a_3-1)$-matching from an $(a_1,a_2,a_3)$-matching; the existence of an $(a_1,a_2+1,a_3-1)$-matching follows by exchanging the roles of $M_1$ and $M_2$.

\subsection{Auxiliary results}
To avoid cluttering the statements, we assume throughout the rest of the section that $G$ is a bipartite graph on $2n$ vertices which is the union of three disjoint perfect matchings $M_1,M_2$, and $M_3$. For convenience, we sometimes refer to the different perfect matchings as colors. Our idea for proving \cref{lem:main} is based on the augmentation of a matching $M$ along a path. For us to be able to do so effectively, we need said path to interact with~$M$, as well as with~$M_1$, $M_2$, and $M_3$, in a controlled way. The following definition captures the interaction of a path with $M$ that we are looking for.
\begin{definition}[nearly-$M$-alternating paths]\label{def:nearly-alternating}
    Let $M$ be a matching of size $n-1$. A path $P=(v_1,\dots,v_\ell)$ is called \emph{nearly-$M$-alternating} if the two $M$-unsaturated vertices of $G$ are on $P$ and the matching $M\cap P$ saturates the remaining vertices of $P$.  
\end{definition}

\begin{observation}\label{obs:unsaturated-positions}
The length $\ell-1$ of a nearly-$M$-alternating path $P = (v_1,\ldots , v_\ell)$ is odd. 
If~$v_i$ and~$v_j$ are the two $M$-unsaturated vertices and $1\leq i < j\leq \ell$, then $i$ is odd and $j$ is even.
\end{observation}

In our attempt to find an $(a_1+1, a_2, a_3-1)$-matching given an $(a_1,a_2, a_3)$-matching $M$, we will first identify an appropriate nearly-$M$-alternating path and then try to shift around the two unsaturated vertices on it in order to achieve that the number of edges of each color in the matching is appropriate. For this the following definition will be useful. 

\begin{definition}\label{def:Mij}
    For a nearly-$M$-alternating path $P=(v_1, \ldots, v_\ell)$, we denote its unique perfect matching by $M_P$. For an odd integer $i'$ and an even integer $j'$ with 
    $1\leq i' < j' \leq \ell$, we define $$M(P;i',j'):= (M_P \Delta  P_{i',j'}) \cup (M\setminus P).$$ 
    \end{definition}

\begin{observation}\label{obs:Mij}
   Let $P=(v_1, \ldots, v_\ell)$ be a nearly-$M$-alternating path, $i'$ be an odd integer, and~$j'$ be an even integer such that $1\leq i' < j' \leq \ell$. 
   Then $M':=M(P;i',j')$ is a matching of size~$n-1$ and $P$ is a nearly-$M'$-alternating path with unsaturated vertices $v_{i'}$ and $v_{j'}$. Furthermore $M' = (M'\cap P) \cup (M'\setminus P)$ and $M'\cap P = M_P\Delta P_{i',j'}$.
\end{observation}  
 \begin{proof}
     To see that the set $M'$ is a matching, note that  $P$ is nearly-$M$-alternating, implying in particular that a vertex of $P$ is $M$-saturated if and only if it is $M\cap P$-saturated. Therefore,~$M\setminus P$ is vertex-disjoint from $P$. Now, since $i'$ is odd and $j'$ is even, the matching  $M_P$ contains the edges $v_{i'}v_{i'+1}$ and $v_{j'-1}v_j$ (and does not contain $v_{i'-1}v_{i'}$ and $v_{j'}v_{j+1}$). Additionally,  $M_P\cap P_{i',j'}$ is a perfect matching on $P_{i',j'}$ and hence $P_{i',j'}\setminus M_P$ is a matching that saturates every vertex of $P_{i',j'}$ except for the endpoints. Similarly, $M_P\setminus P_{i',j'}$ saturates every vertex not on $P_{i',j'}$. Thus,  $M_P \Delta  P_{i',j'}$ is a matching on $P$ saturating every vertex of $P$ except for $v_{i'}$ and $v_{j'}$, which together with the edges in $M\setminus P$ forms a matching of size~$n-1$. Moreover, $P$ is nearly-$M'$-alternating, as needed.
 \end{proof}

   Note that, if the $M$-unsaturated vertices of $P$ in the above observation are $v_i$ and $v_j$, where~$1\leq i < j \leq \ell$, then $M=M(P;i,j)$.
   \smallskip

Throughout our proofs, we will be tracking how the shifting of unsaturated vertices along a nearly-$M$-alternating path $P$ affects the number of edges  of each color in the matching. We define a function that will allow us to do this  by looking only at the edges of a specific color between the two unsaturated vertices of $P$. 

\begin{definition}[$f_c(P,M)$]
    Let $M$ be a matching of size $n-1$ and $P$ be a nearly-$M$-alternating path with $M=M(P;i,j)$ for some $i<j$. Given $c\in\{1,2,3\}$, define $$f_c(P,M) := |(P_{i,j}\cap M_c)\cap M|-|(P_{i,j}\cap M_c)\setminus M|.$$
\end{definition}

\begin{observation}\label{obs:changes}
 Let $M$ and $M'$ be matchings of size $n-1$ and $P$ be a path that is both nearly-$M$-alternating and nearly-$M'$-alternating and such that $M\setminus P = M'\setminus P$. Then, for any $c\in \set{1,2,3}$, we have
  $$      a_c(M') - a_c(M) = f_c(P,M') - f_c(P,M).
    $$  
 \end{observation}
\begin{proof} 
     By Observation~\ref{obs:Mij}, if $M=M(P;i,j)$, then we have $M\cap P = M_P\Delta P_{i,j}$  and consequently $a_c(M\cap P) = a_c(M_P)+f_c(P,M)$. Similarly, $a_c(M'\cap P)=a_c(M_P)+f_c(P,M')$. As $M\setminus P = M'\setminus P$, we get
    \begin{align*}
        a_c(M') = a_c(M'\setminus P) + a_c(M'\cap P) &= a_c(M'\setminus P)+a_c(M_P)+ f_c(P,M')\\ &= a_c(M\setminus P)+a_c(M\cap P)-f_c(P,M)+ f_c(P,M')\\
        &=a_c(M)-f_c(P,M)+ f_c(P,M').
    \end{align*}
    Rearranging gives the desired equality.
\end{proof}

\bigskip
Next, we define what we consider desirable interaction of a path with the matchings $M_1$, $M_2$, and $M_3$.
\begin{definition}[good paths]
   Let $c\in \set{1,2}$. A path $P = (v_1,\dots,v_\ell)$ is \emph{$c$-good} if 
   \begin{align*}
       P\cap M_c \subseteq \set{v_{2s+1}v_{2s+2}\,:\, 0\leq s\leq (\ell-2)/2} \,\,\text{and} \,\, P\cap M_{3-c} \subseteq \set{v_{2s}v_{2s+1}\,:\, 1\leq s\leq (\ell-1)/2}, 
   \end{align*}
   that is, if all $M_c$-edges appear in odd positions and all $M_{3-c}$-edges appear in even positions on $P$.
   We say that $P$ is \emph{good} if it is $c$-good for some $c\in \set{1,2}$.
\end{definition}

In other words, a path $P$ is $c$-good if  $P$ is $(M_c\cup M_3,M_{3-c}\cup M_3)$-alternating and the first edge~$v_1v_2$ of $P$ is in $M_c\cup M_3$. 
Note that an $M_1$- or $M_2$-alternating path is always good, but the notion of a good path provides some extra flexibility by allowing the path to contain edges from $M_3$ as well. Often we look at paths which are nearly-$M$-alternating for some given matching $M$ and good at the same time. We make an important observation about the meaning of the functions $f_1$ and $f_2$ on such paths.

\begin{observation}\label{obs:f-good-paths}
    Let $P$ be a $c$-good, nearly-$M$-alternating path for some color $c\in \set{1,2}$ and matching $M=M(P;i,j)$ of size $n-1$. Then the number of edges of $P_{i,j}$ in color $c' \in \{ 1,2\}$ 
    is~$-f_{c'}(P,M)$ if $c'=c$ and $f_{c'}(P,M)$ if $c'\neq c$. 
\end{observation}
\begin{proof}
    By definition, 
    \begin{align*}
    f_{c'}(P,M) = |(P_{i,j}\cap M_{c'})\cap M|-|(P_{i,j}\cap M_{c'})\setminus M|.
    \end{align*}
    Since $P$ is $c$-good and $P_{i,j}$ is $M$-alternating, it follows that $|P_{i,j}\cap M_{c'}\cap M|=0$ if $c'= c$ and $|(P_{i,j}\cap M_{c'})\setminus M|=0$ if $c'\neq c$. Therefore, $P_{i,j}$ contains exactly $\size{f_{c'}(P,M)}$ edges of $M_{c'}$. The claim then follows from the fact the only surviving term in the definition of $f_{c'}(P,M)$ is nonpositive in the former case and nonnegative in the latter. 
\end{proof}

The proof of \cref{lem:main} attempts, as much as possible, to shift the unsaturated vertices of our carefully chosen nearly-$M$-alternating path towards its end, while keeping the number of edges of color~$2$ in the matching unchanged. The next lemma describes one step in this procedure and identifies  when such a shift is not  possible anymore. 

\begin{lemma}\label{lem: can move_n}
   Let $M\subseteq G$ be a matching and $P=(v_1,\dots,v_\ell)$ be a good, nearly-$M$-alternating path such that $M=M(P;i,j)$ for some $j\leq \ell-2$ even and $i<j$ odd. Assume further that either $i\leq j-3$, or~$i+1=j$ and $v_{i}v_{j}\notin M_3$. Then, one of $M(P;i+2,j),M(P;i,j+2)$, or $M(P;i+2,j+2)$ is a matching $M'$ with $a_2(M) = a_2(M')$ and $|a_3(M)-a_3(M')|\leq 1$.
\end{lemma}
\begin{proof}
    All matchings in this proof are with respect to $P$ and $M$, so we will write $M_{s,t}$ for the matching $M(P;s,t)$. We also refer to $M$ as $M_{i,j}$ to highlight the interaction of $M$ with $P$, but it is important to keep in mind that these are the same matching.
    We may assume that one of $v_{j}v_{j+1}$ and $v_{j+1}v_{j+2}$ is in $M_2$, as otherwise  we are done by taking $M_{i,j+2}=M_{i,j}\setminus \set{v_{j+1}v_{j+2}}\cup\{v_{j}v_{j+1}\}$. 
        
    Suppose first that $i\leq j-3$. Then, by the same argument,  one of $v_{i}v_{i+1}$ and $v_{i+1}v_{i+2}$ is in~$M_2$, as otherwise $M_{i+2,j}=M_{i,j}\setminus\set{v_{i+1}v_{i+2}}\cup\{v_{i}v_{i+1}\}$ satisfies the required properties. Let $M' = M_{i+2,j+2}=M_{i,j}\setminus\set{v_{i+1}v_{i+2},v_{j+1}v_{j+2}} \cup \{v_{i}v_{i+1},v_{j}v_{j+1}\}$. Since $P$ is a good path, $i$ is odd, and $j$ is even, we know that either $v_{i}v_{i+1},v_{j+1}v_{j+2}\in M_2$ or $v_{i+1}v_{i+2},v_{j}v_{j+1}\in M_2$. Either way, it follows that $a_2(M)=a_2(M')$, since we add exactly one edge of~$M_2$ but also remove exactly one edge of~$M_2$. At most one of the removed edges and at most one of the added edges is in $M_3$, and thus we also have $|a_3(M)-a_3(M')|\leq 1$.
            
    Suppose now that $i = j-1$. By assumption,  $v_{i}v_j\notin M_3$. Recall that one of $v_jv_{j+1}$ and $v_{j+1}v_{j+2}$ belongs to $M_2$. Therefore, since $P$ is good, either $v_{i}v_{j},v_{j+1}v_{j+2}\in M_2$ or $v_{j}v_{j+1}\in M_2$. Either way, $M' = M_{i+2,j+2} = M_{i,j}\setminus\set{v_{j+1}v_{j+2}}\cup\{v_iv_j\}$ satisfies the statement.
\end{proof}

We now apply this step-wise change iteratively to get a sequence of matchings ``deforming'' one matching into another in ``small'' steps.
While it is quite simple to keep track of how the number of $M_i$-edges changes in \cref{lem: can move_n}, this becomes more difficult when we apply \cref{lem: can move_n} multiple times. The next lemma looks at a sequence, generated by  applying \cref{lem: can move_n} iteratively, and, similarly to the Intermediate Value Theorem, allows us to pick a matching with a specific number of $M_3$-edges in this sequence.

\begin{lemma} (Intermediate Value Lemma)\label{lem: intermediate value_n}
 Let $M$ and $M'$ be matchings of size $n-1$ and $P$ be a good path that is both nearly-$M$-alternating and nearly $M'$-alternating and such that $M\setminus P = M'\setminus P$. 
 Assume furthermore that $a_2(M) = a_2(M') =:a_2$ and $a_3(M)\leq a_3(M')$. Then, for any integer $a_3^*$ satisfying $a_3(M)\leq a_3^*\leq a_3(M')$, there exists an $(n-1-a_2-a_3^*,a_2,a_3^*)$-matching in $G$.
\end{lemma}

\begin{proof}
    Let $P=(v_1,\ldots,v_\ell)$.    
     As $P$ is both nearly-$M$-alternating and nearly $M'$-alternating and~$M\setminus P = M'\setminus P$, we may write $M = M(P;i,j)$ and $M' = M(P;i',j')$ for some even $j,j'\leq \ell$ and odd~$i<j$ and~$i'<j'$. As in the proof of \cref{lem: can move_n}, all matchings in this proof are with respect to $P$ and $M$, and we write $M_{s,t}$ for $M(P;s,t)$. To prove the statement, we find a sequence of matchings of size $n-1$ each, with the first matching being $M_{i,j}$ and the last being $M_{i',j'}$, such that each matching in the sequence has exactly $a_2$ edges of $M_2$ and the numbers of $M_3$-edges in two consecutive matchings differ by at most one. This suffices to prove the statement, since there must exist a matching in this sequence with exactly $a_3^*$ edges of $M_3$. 
    We consider three different cases, in each of which we find a sequence of matchings as described above.
    \smallskip

    \noindent\underline{Case 1}: One of the two intervals $[i,j]$ and $[i',j']$ contains the other; without loss of generality, assume $i\leq i'<j'\leq j$. 
    Since $P$ is good and $a_2(M_{i,j})=a_2(M_{i',j'})$, by \cref{obs:changes}, $f_2(P,M_{i,j})=f_2(P,M_{i',j'})$ and hence, by \cref{obs:f-good-paths}, it follows that $P_{i,j}$ and $P_{i',j'}$ contain the same number of $M_2$-edges. Therefore, neither $P_{i,i'}$ nor $P_{j',j}$ contains an edge of $M_2$. 
    Consider then the sequence of matchings $M_{i,j}, M_{i+2,j},\dots, M_{i',j}, M_{i',j-2},\dots, M_{i',j'}$.
    In this sequence, each matching is obtained from the previous one by exchanging one edge for another and thus, the numbers of $M_3$-edges of two consecutive matchings naturally differ by at most one. Furthermore, none of the exchanged edges are in $M_2$ and hence, every matching in the sequence contains exactly $a_2(M_{i,j})$ edges of $M_2$, as required.

    \smallskip

    \noindent\underline{Case 2}: $P_{i,j}$ does not contain any edges of $M_2$. By \cref{obs:changes,obs:f-good-paths}, neither does $P_{i',j'}$. Consider the sequence  of matchings $M_{i,j},M_{i+2,j},\dots,M_{j-1,j},M_{j'-1,j'},M_{j'-3,j'},\dots, M_{i',j'}$. Again, each matching is obtained from the previous one by exchanging one edge for another and has exactly $a_2(M_{i,j})$ edges of $M_2$.

    \smallskip

    \noindent\underline{Case 3}: Neither of the two cases above happen, that is, neither of the intervals contains the other and $P_{i,j}$ shares an edge with $M_2$. 
    Then, without loss of generality, assume that $i\leq i'$ and $j\leq j'$ (otherwise, reverse the direction of the path and use the fact that the intervals do not contain each other). 
    Let $\{M^t\}_{t=1}^{t^*}$ be the sequence of matchings obtained by iteratively applying \cref{lem: can move_n} starting with $M_{i,j}$ and stopping when $M^{t}=M_{i^*,j^*}$ does not satisfy the hypotheses of \cref{lem: can move_n}, or $i^* = i'$ or $j^* = j'$. The former occurs if $j^* = \ell$, in which case we have $j^* = j' = \ell$, or if $ i^*+1=j^*$ and $v_{i^*}v_{j^*}\in M_3$. Since $P_{i,j}$ contains an edge of $M_2$, by \cref{obs:changes,obs:f-good-paths}, $P_{i^*,j^*}$ contains an edge of $M_2$, as \cref{lem: can move_n} ensures that $a_2(M_{i,j})=a_2(M_{i^*,j^*})$; thus, the latter stopping condition cannot occur. 
    Therefore, the final matching $M^{t^*}=M_{i^*,j^*}$ satisfies $i^* = i'$ and $j^*\leq j'$, or $i^*\leq i'$ and $j^*=j'$. 
    Suppose $i^* = i'$ and $j^*\leq j'$; the other case follows from a similar argument. 
    Consider the sequence $M_{i,j}=M^1,M^2,\ldots,M^{t^*}=M_{i',j^*},M_{i',j^*+2},\dots,M_{i',j'}$. We show that this sequence satisfies the required properties. As mentioned above, \cref{lem: can move_n} guarantees  that each matching in~$\{M^t\}_{t=1}^{t^*}$ contains exactly $a_2(M_{i,j})$ edges of $M_2$ and the numbers of $M_3$-edges of any two consecutive matchings differ by at most one. Let us show that the extension also satisfies these two properties. Observe that each matching in $M_{i',j^*},M_{i',j^*+2},\dots,M_{i',j'}$ is obtained from the previous matching by replacing one edge of $P_{j^*,j'}$ with another. Thus, the numbers of $M_3$-edges in any two consecutive matchings differ by at most one. By \cref{obs:changes,obs:f-good-paths}, $P_{i',j'}$ and~$P_{i',j^*}$ contain the same number of edges of $M_2$. Hence, $P_{j^*,j'}$ does not contain any edges of $M_2$.
    Therefore, each matching in the extension also contains exactly $a_2(M_{i,j})$ edges of $M_2$, completing the proof.
\end{proof}

This concludes the development of our tool-set for analyzing good, nearly-$M$-alternating paths. What is left to do is to find an appropriate path of this type, satisfying several additional properties, to which we apply these tools. In general, our aim is to modify $M$ along a good, nearly-$M$-alternating path $P$ in such a way that the resulting matching has fewer $M_3$-edges. Consequently, $P$ should  contain at least one edge of $M\cap M_3$. 
As a first step, we consider the path components in the graphs~$M\cup M_1$ and $M\cup M_2$. Note that each of these graphs is the disjoint union of even cycles, isolated edges, and possibly a single path of length at least three. Note that, if the two $M$-unsaturated vertices form an isolated-edge component in $M\cup M_1$, then we can easily find an $(a_1+1,a_2,a_3-1)$-matching in $G$ using this edge. Thus, we may assume that these vertices belong to the unique nontrivial path component in $M\cup M_1$ and moreover, they must be its endpoints. Additionally, the two $M$-unsaturated vertices must be the endpoints of a path component (possibly consisting of a single edge) in $M\cup M_2$.

\begin{definition}[$P_i(M)$, $C_0(M)$]
    Given $i\in\{1,2\}$, a matching $M$ of size $n-1$, let~$P_i(M)$ denote the unique component in the graph $M\cup M_i$, containing the two $M$-unsaturated vertices.  We additionally write $C_0(M)$ for the graph $P_1(M)\cup P_2(M)$.
\end{definition}

Note that the components $P_1(M)$ and $P_2(M)$ can intersect in $M_3$-edges belonging to $M$. In the next lemma we argue that, if $G$ contains no $(a_1+1,a_2,a_3-1)$-matching, then the graph~$C_0(M)$ must have a very special structure.

\begin{lemma}\label{lem:M3edge}
    Let $a_1,a_2,a_3\in \mathbb{N}_0$ be integers with $a_1+a_2+a_3 = n-1$. Suppose $G$ contains an~$(a_1,a_2,a_3)$-matching $M$ but does not contain an $(a_1+1,a_2,a_3-1)$-matching. Then  $C_0(M)$ is an~$(M_1,M_2)$-alternating cycle.
\end{lemma}
\begin{proof}
For each $i\in \{1,2\}$, set $P_i=P_i(M)$. Denote the two $M$-unsaturated vertices by $u_1$ and~$u_2$. Suppose $P_1\cup P_2$ contains an $M_3$-edge.  Let $\ell_i\geq 1$ denote the number of $M_{i}$-edges between $u_i$ and the  $M_3$-edge closest to it on $P_i$, where we set $\ell_i$ to be infinity if no such edge exists. If $\ell_1$ and $\ell_2$ are both infinite, then neither  $P_1$ nor $P_2$ contains an edge of $M_3$. Thus, all the interior vertices of~$P_1$, respectively~$P_2$, are saturated by $M\cap M_2$, respectively $M\cap M_1$. Therefore, these two paths only intersect in their endpoints. Hence, $P_1\cup P_2$ forms a cycle in $M_1\cup M_2$ and we are done. Thus, suppose that at least one of $\ell_1$ and $\ell_2$ is finite. We will show that this implies that $G$ contains an~$(a_1+1,a_2,a_3-1)$-matching, leading to the desired contradiction.
We consider two cases. 

\noindent\underline{Case 1}: $\ell_1\leq \ell_2$. Let $e_1$ be the first $M_3$-edge on $P_1$ (counting from $u_1$). Let $P_1'$ denote the sub-path of $P_1$ starting at $u_1$ and ending after passing through $e_1$. Observe that $P_1'$ has length $2\ell_1$ and $M'=M\Delta P_1'$ is an $(a_1+\ell_1,a_2-\ell_1+1,a_3-1)$-matching. For finite $\ell_2$, let $P_2'$ denote the sub-path of $P_2$ of length $2(\ell_1-1)$ starting at $u_2$. Since $\ell_2\geq \ell_1$, we know that $P_2'$ is $(M_2,M\cap M_1)$-alternating and hence, every vertex of $P_2'$ other than $u_2$ is incident to an edge of $M\cap M_1$. Therefore, $P_2'$ does not intersect $P_1'$ and $M''=M'\Delta P_2'$ is an $(a_1+1,a_2,a_3-1)$-matching, as required. So, suppose that $\ell_2$ is infinite. Then, it may happen that the length of $P_2$ is less than $2(\ell_1-1)$ and we cannot simply take $P_2'$ to be a sub-path of $P_2$. However, since $u_1$ is the second endpoint of $P_2$, we may continue~$P_2$ by going through $u_1$ and following the edges of~$P_1$ until (and not including) $e_1$. Note that this extension of $P_2$ is an $(M_2,M'\cap M_1)$-alternating path of length at least $2\ell_1$. Let $P_2'$ be the sub-path  of length $2(\ell_1-1)$ starting at $u_2$. Finally, we get that $M''= M'\Delta P_2'$ is the desired $(a_1+1,a_2,a_3-1)$-matching. 

\noindent\underline{Case 2}: $\ell_2 <\ell_1$. We proceed similarly as in the previous case, first taking the symmetric difference of $M$ with the first $2\ell_2$ edges along $P_2$ starting from $u_2$ to get an $(a_1-\ell_2+1,a_2+\ell_2,a_3-1)$-matching $M'$, and then taking the symmetric difference of $M'$ with the first $2\ell_2$ edges along $P_1$ starting from~$u_1$, where we again extend $P_1$ through the beginning of $P_2$ if $\ell_1$ is infinite.
\end{proof}
The above lemma shows  that $P_i(M)$ is a good candidate for a path along which we can augment~$M$, unless $C_0(M)$ is a cycle of $M_1\cup M_2$. In the latter, case we will attempt to find a different path. If $C_0(M)$ is a cycle of $M_1\cup M_2$, then we must find a path $P$ from $u_1$ leaving this cycle, as otherwise $P$ does not contain an edge of $M_3$. 
The next simple result gives us such an $M$-alternating path but does not give us any information about this path with respect to $M_1$, $M_2$, and $M_3$.

\begin{lemma}\label{lem: alternating path_n}
    Let $M\subseteq G$ be a matching of size $n-1$ and $v\in V(G)$ be a vertex unsaturated by~$M$. Then every vertex of $G$ can be reached from $v$ by an $M$-alternating path.
\end{lemma}

    \begin{proof}
        Let $A\sqcup B$ be the bipartition of $G$ and suppose $v\in A$.
        Since $G$ is $3$-regular, we know that $\size{A} = \size{B}$ and hence, every vertex in $A\setminus\set{v}$ is saturated by $M$. Let $A'\subseteq A$ and $B'\subseteq B$ be the sets of vertices reachable from $v$ via $M$-alternating paths and note that $A'\neq \emptyset$, as $v\in A'$. Observe that, for each $w\in B'$ that is saturated by $M$,  its matching partner $M(w)$ belongs to $A'$, since any $M$-alternating path from $v$ to $w$ does not contain $M(w)$ and can thus be extended to $M(w)$. But~$B'$ contains at most one vertex that is unsaturated by $M$ and $v\in A'$ is unsaturated by $M$, implying~$|A'|\geq |B'|$. 
        Suppose that $ (A\setminus A')\cup  (B\setminus B')$ is nonempty. Since $G$ is connected, there exists an edge $xy$ with $x\in A'\cup B'$ and $y\in (A\setminus A')\cup  (B\setminus B')$. As a result, using the fact that~$|A'|\geq |B'|$, we conclude that there are at most $3|A'|-1$ edges between $A'$ and $B'$; indeed, $G$ is 3-regular and at least one edge incident to a vertex of $A'$ does not go to $B'$. Therefore, we may assume that~$x\in A'$ and $y\in B\setminus B'$. 
        The edge $xy$ cannot be in $M$, as the only way to reach $x$ on an $M$-alternating path is via $M(x)$. Thus, taking an arbitrary $M$-alternating path from $v$ to $x$ and adding the edge $xy$, we reach $y$ on an $M$-alternating path, a contradiction. 
    \end{proof}

The final lemmas of this section cleverly apply \cref{lem: alternating path_n} to find the desired path while maintaining some control over its interaction with $M_1$, $M_2$, and $M_3$. We begin with a more technical lemma, from which we will then derive the final result about the existence of a suitable path. We first state the following simple observation about matchings in even cycles, phrased in a language convenient for our subsequent use.
\begin{observation}\label{obs: matchings of a cycle_n}
    Let $C$ be a cycle in $M_1\cup M_2$, $v\in V(C)$, $c\in\{1,2\}$, and $k< |C|/2$ be an integer. Then there exists a matching $M$ in $C$ with $|M\cap M_c| = k$ and $|M\cap M_{3-c}| = |C|/2-1-k$ such that~$v$ is unsaturated by $M$.
\end{observation}
\begin{proof}
    Suppose $C=(v_1,v_2,\ldots,v_{|C|})$ with $v=v_1$ and $v_1v_2\in M_{3-c}$. Then, the desired matching~$M$ can be obtained by adding the first $k$ edges of $M_c$ and the last $|C|/2-1-k$ edges of  $M_{3-c}$ to $M$.
 \end{proof}

In addition to the cycle $C_0(M)$, we consider another collection of cycles in $M_1\cup M_2$, on which we cannot augment $M$.
\begin{definition}[$\mathcal{C}(M)$]
Let $M$ be a matching. We define $\mathcal{C}(M)$ to be the collection of cycles in $(M\cap M_2)\cup M_1$ and $(M\cap M_1)\cup M_2$. In other words, $\mathcal{C}(M)$ contains all cycles of the graph $M_1\cup M_2$ that are also $M$-alternating.
\end{definition}

We are now ready to state our first (technical) lemma, establishing the existence of an $M$-alternating path $P_0$ from $u_1$  containing a matched edge of $M_3$. For this, we need $P_0$ to leave~$C_0(M)\cup \mathcal{C}(M)$ altogether. 
 
\begin{lemma}\label{lem:existence P_0}
    Let $a_1,a_2,a_3\in \mathbb{N}_0$ be integers with $a_1+a_2+a_3 = n-1$ and $a_3>0$. Suppose $G$ contains an $(a_1,a_2,a_3)$-matching but does not contain an $(a_1+1,a_2,a_3-1)$-matching. Then there exists an $(a_1,a_2,a_3)$-matching $M\subseteq G$ with unsaturated vertices $u_1$ and $u_2$ such that $C_0(M)$ is a cycle in $M_1\cup M_2$, and an $M$-alternating path $P_0 = (v_1,\dots, v_{\ell+1})$ such that:
    \begin{enumerate}[label=(\roman*)]
        \item $v_1=u_1$, $v_{\ell},v_{\ell+1}\in V(G)\setminus (V(C_0(M))\cup V(\mathcal{C}(M)))$, and all other vertices of $P_0$ are in~$V(\mathcal{C}(M))$.\label{lem:existence P_0:vertices}
        \item $v_{\ell}v_{\ell+1} \in M$.\label{lem:existence P_0:last edge}
        \item $V(P_0) \cap V(C_0(M)) = \set{u_1}$.\label{lem:existence P_0:C_0}
        \item For every cycle $C\in \mathcal{C}(M)$, the vertices in $V(C)\cap V(P_0)$ form a sub-path in both $P_0$ and $C$.\label{lem:existence P_0:intervals}
        \item $M\cap P_0 \subseteq M_1$ or $M\cap P_0 \subseteq M_2$. \label{lem:existence P_0:matched edges}
    \end{enumerate}
\end{lemma}
\begin{proof}
    We proceed in two steps. First, we show that we can find a pair $(M,P_0)$ satisfying properties~\ref{lem:existence P_0:vertices}--\ref{lem:existence P_0:intervals}, and then we argue that there exists such a pair satisfying \ref{lem:existence P_0:matched edges} as well. 

    For the first step, fix a pair $(M,P)$, where $M$ is an $(a_1,a_2,a_3)$-matching in $G$ with unsaturated vertices $u_1$ and $u_2$ and $P$ is an $M$-alternating path from $u_1$ to a vertex $v\notin V(C_0(M))\cup V(\mathcal{C}(M))$ such that 
    \begin{enumerate}[label=(\arabic*)]
        \item\label{P0:switches} the number of $M_3$-edges in $P$ is minimized among all pairs $(M,P)$;
        \item\label{P0:length} subject to~\ref{P0:switches}, the length of $P$ is minimized.
    \end{enumerate}
    Before we prove that a slight extension of $(M,P)$ satisfies properties~\ref{lem:existence P_0:vertices}--\ref{lem:existence P_0:intervals}, let us show that there exists a pair $(M,P)$ satisfying the above conditions. By assumption, there exists some $(a_1,a_2,a_3)$-matching $M\subseteq G$. Let $u_1,u_2$ be its unsaturated vertices. By \cref{lem:M3edge}, we may assume that~$C_0(M)$ is a cycle in $M_1\cup M_2$. Since $a_3>0$, the matching $M$ contains at least one $M_3$-edge; the endpoints of such an edge are not contained in $V(C_0(M))\cup V(\mathcal{C}(M))$, so $V(G)\setminus \parens*{V(C_0(M))\cup V(\mathcal{C}(M))}$ contains at least one vertex~$v$. By \cref{lem: alternating path_n}, there exists an $M$-alternating path from $u_1$ to $v$, showing that there exists at least one candidate pair $(M,P)$. So, fix a pair $(M,P)$ satisfying conditions~\ref{P0:switches} and~\ref{P0:length}. Next, we prove that we can extend $P$ by one  edge to obtain a path $P'$ such that $(M,P')$ satisfies~\ref{lem:existence P_0:vertices}--\ref{lem:existence P_0:intervals}.

    Write $P = (v_1,\dots, v_{\ell})$. 
    Note that condition~\ref{P0:length} implies that $v_1,v_2,\dots,v_{\ell-1}\in V(C_0(M))\cup V(\mathcal{C}(M))$ as otherwise, we could replace $P$ with a sub-path of itself. Since $v_{\ell-1}\in V(C_0(M))\cup V(\mathcal{C}(M))$ and $v_\ell\notin V(C_0(M))\cup V(\mathcal{C}(M))$, it follows that $v_{\ell-1}v_{\ell}\in M_3\setminus M$. Set $v_{\ell+1} = M(v_{\ell})$ and let $c\in \{1,2\}$ be such that $v_{\ell}v_{\ell+1}\in M\cap M_c$. Define $P' = (v_1,\dots, v_{\ell+1})$, which satisfies property~\ref{lem:existence P_0:last edge} by construction. Observe that, for any saturated vertex $v\in V(C_0(M))\cup V(\mathcal{C}(M))$, we have $M(v)\in V(C_0(M))\cup V(\mathcal{C}(M))$, so $v_{\ell+1}\notin V(C_0(M))\cup V(\mathcal{C}(M))$, as needed. 
    
    Now, we verify property~\ref{lem:existence P_0:C_0}, which also completes the justification of property~\ref{lem:existence P_0:vertices}. Suppose that $P'$ intersects $C_0(M)$ in some vertex $v_i\neq u_1$. Consider the matching $M\cap C_0(M)$, which has size $e(C_0(M))/2-1$ and contains $0\leq k \leq e(C_0(M))/2-1$ edges of $M_1$. By Observation~\ref{obs: matchings of a cycle_n}, there exists another matching $M'$ of $C_0(M)$ of size $e(C_0(M))/2-1$ that also has $k$ edges of $M_1$ but leaves $v_i$ unsaturated. Then, $M'' = (M\setminus C_0(M))\cup M'$ is also an $(a_1,a_2,a_3)$-matching with $C_0(M'') = C_0(M)$ and $\mathcal{C}(M) = \mathcal{C}(M'')$. Moreover, $(v_i,\dots, v_{\ell})$ is a shorter $M''$-alternating path from an unsaturated vertex of $M''$ to a vertex not in $V(C_0(M''))\cup V(\mathcal{C}(M''))$ and with at most as many $M_3$-edges as $P$ since it is a sub-path of $P$. Thus, we get a contradiction to our choice of $(M,P)$.

    It remains to prove that $(M,P')$ also satisfies property~\ref{lem:existence P_0:intervals}. Suppose there is a cycle $C\in \mathcal{C}(M)$ violating \ref{lem:existence P_0:intervals}. Let $v_i$ and $v_j$ be the first and the last vertex of $C$ appearing in $P'$, respectively. Assume that the sub-path $(v_i,\dots, v_j)$ of $P'$ between $v_i$ and $v_j$ is not a sub-path of $C$. Let $L_1$ and~$L_2$ be the two paths between $v_i$ and $v_j$ along $C$. Note that $e(L_1), e(L_2)$, and $j-i$ all have the same parity since $G$ is bipartite. Additionally, $P'$ can only leave $C$ and return to $C$ on $M_3$-edges that are not part of $M$, so the path $(v_i,\dots, v_j)$ contains at least one $M_3$-edge. Since $v_i$ is incident to an edge of $M$ in exactly one of $L_1$ or $L_2$, it follows that either $(v_1,\dots v_i, L_1,v_j,v_{j+1},\dots, v_{\ell})$ or $(v_1,\dots v_i, L_2,v_j,v_{j+1},\dots, v_\ell)$ is an alternating path with fewer $M_3$-edges, again contradicting our choice of $(M,P)$. 

    \medskip
    Thus, we have established the existence of a pair $(M,P')$ satisfying properties~\ref{lem:existence P_0:vertices}--\ref{lem:existence P_0:intervals}. We will show that we can choose a pair $(M,P_0)$ so that property~\ref{lem:existence P_0:matched edges} holds as well. Among all pairs $(M,P_0)$ satisfying properties~\ref{lem:existence P_0:vertices}--\ref{lem:existence P_0:intervals}, where $P_0=(v_1,\dots, v_{\ell+1})$ and $v_\ell v_{\ell+1}\in M\cap M_c$ for some $c\in\{1,2\}$, fix a pair such that 
    \begin{enumerate}[label=(\arabic*)]
    \setcounter{enumi}{2}
        \item\label{P_0:M_3} the number of $M_3$-edges in $P_0$ is minimized;
        \item\label{P_0:length} subject to~\ref{P_0:M_3}, the last edge in $P_0$ from $M\cap M_{\bar{c}}$ appears as early as possible, where $\bar{c}=3-c$ and we consider it the ``earliest'' if there is no such edge at all.
    \end{enumerate}
    Let $C_1,\dots, C_r\in \mathcal{C}(M)$ be the cycles visited by $P_0$, in the order in which they appear, and note that, for each $i\in [r]$, we have $C_i\cap M\subseteq M_c$ or $C_i\cap M\subseteq M_{\bar{c}}$. Note that property~\ref{lem:existence P_0:intervals} implies that all $C_1,\dots, C_r$ are distinct. Moreover, each edge in $P_0$ entering or leaving a cycle $C_i$ is in $M_3$ and, by \ref{lem:existence P_0:vertices} and the fact that $v_\ell v_{\ell+1}\in M\cap M_c$, these are the only $M_3$-edges in $P_0$. It follows that $P_0$ contains exactly $r+1$ edges of $M_3$.

    Let $j\in [r]$ be the largest index such that  $C_j\cap M\subseteq M_{\bar{c}}$. Since $C_j\cap M\cap P_0$ contains at least one edge, the last edge of $P_0$ in $M\cap M_{\bar{c}}$ is on $C_j$. Let $k$ denote the number of edges in $M\cap M_c\cap C_0(M)$, i.e., $k = (e(P_{\bar{c}}(M))-1)/2$ (recall that $C_0(M) = P_{{c}}(M)\cup P_{\bar{c}}(M)$). Then, $M\cap C_0(M)$ contains precisely $e(C_0(M))/2-k-1$ edges of $M_{\bar{c}}$. Next, we show that, by adjusting $M$ on $C_0(M)$ and $C_j$, we can find a new matching and alternating path that contradict the choice of the pair $(M, P_0)$. Let us consider two cases depending on the length of $C_j$.

    \noindent \underline{Case 1}: $e(C_j)/2\leq k$. By \cref{obs: matchings of a cycle_n}, there exists a matching~$M_{C_0}$ in $C_0(M)$ with $e(C_0(M))/2-k-1+e(C_j)/2$ edges of $M_{\bar{c}}$ and $k-e(C_j)/2$ edges of $M_{c}$ that also leaves~$u_1$ unsaturated. Note that~$M_{C_0}$ contains exactly $e(C_0(M))/2-1$ edges. Furthermore, $(M\setminus C_0(M))\cup M_{C_0}$ contains $e(C_j)/2$ more edges of $M_{\bar{c}}$ and $e(C_j)/2$ fewer edges of $M_c$ than $M$.
    Let $M' = ((M\setminus C_0(M))\cup M_{C_0})\Delta C_j$. Then, $M'$ is also an $(a_1,a_2,a_3)$-matching and $P_0\Delta C_j$ is an $M'$-alternating path with the same number of $M_3$-edges but the last edge of $P_0$ in $M'\cap M_{\bar{c}}$ either does not exist or appears before $C_j$, contradicting the choice of $(M,P_0)$.

    \noindent \underline{Case 2}: $e(C_j)/2 > k$. Observe that $(M\setminus C_0(M))\cup (C_0(M)\cap M_{\bar{c}})$ is a perfect matching containing~$k+1$ more edges of $M_{\bar{c}}$ and $k$ fewer edges of $M_c$ than $M$.
    Let $1\leq h\leq \ell$ be maximal such that~$v_h\in V(C_j)$. Let $M_{C_j}$ denote the matching of $C_j$ with $e(C_j)/2-k-1$ edges of $M_{\bar{c}}$ and $k$ edges of~$M_{c}$ such that $v_h$ is unsaturated, which again exists by \cref{obs: matchings of a cycle_n}. Note that $M_{C_j}$ contains $e(C_j)/2-1$ edges.
    Now, consider $M' = (M\setminus (C_0(M)\cup C_j))\cup (C_0(M)\cap M_{\bar{c}})\cup M_{C_j}$. Indeed, $M'$ is an $(a_1,a_2,a_3)$-matching, since the flipped edges on $C_j$ balance out the flipped edges on $C_0(M)$. Furthermore, $C_0(M') = C_j$, $\cal C(M') = \cal C(M)\setminus \set{C_0(M)}\cup \set{C_j}$, and  $v_h$ is unsaturated in $M'$. Setting  $P_0' = (v_h,\dots,v_{\ell+1})$, we get a contradiction to our choice of $(M,P_0)$, as $P_0'$ is an $M'$-alternating path with fewer $M_3$-edges while still satisfying properties \ref{lem:existence P_0:vertices}--\ref{lem:existence P_0:intervals}.
\end{proof}

From \cref{lem:existence P_0} we will now deduce the existence of a very special nearly-$M$-alternating path, along which we will modify our matching.

\begin{lemma}\label{lem:existence P}
    Let $a_1,a_2,a_3\in \mathbb{N}_0$ be integers with $a_1+a_2+a_3 = n-1$ and $a_3>0$. Suppose~$G$ contains an $(a_1,a_2,a_3)$-matching but does not contain an $(a_1+1,a_2,a_3-1)$-matching. Then there exists an $(a_1,a_2,a_3)$-matching $M\subseteq G$ with unsaturated vertices $u_1$ and $u_2$ and a nearly-$M$-alternating $c$-good path $P = (w_1,\dots, w_{k})$ such that:
    \begin{enumerate}[label=(\alph*)]
        \item $w_1=u_2$, $w_h=u_1$ for some even $h$ such that $1 < h < k$, and $M = M(P;1,h)$;\label{existence P-unsaturated}
        \item $P_{1,h}$ is $(M_1,M_2)$-alternating and in particular $f_3(P, M) = 0$;\label{existence P-no M3 between unmatched}
        \item There exists an odd $t$ with $h < t < k-2$ such that $w_{t-1}w_t\in M_3\setminus M$ and $w_tw_k\in  M_{{3-c}}$;\label{existence P-cycle}
        \item The path $(w_1,\dots, w_{t})$ is $M_c$-alternating;\label{existence P-Mc alternating}
        \item The cycle $(w_t,\dots, w_k, w_t)$ is $(M_{3-c},M)$-alternating and contains an edge of $M_3\cap M$.\label{existence P-matched M3 edge}
    \end{enumerate}
\end{lemma}
\begin{proof}
    Suppose that $G$ does not contain an $(a_1+1,a_2,a_3-1)$-matching. Then, by \cref{lem:existence P_0}, there exists an $(a_1,a_2,a_3)$-matching $M\subseteq G$ with unsaturated vertices $u_1$ and~$u_2$ and an $M$-alternating path $P_0\subseteq G$ with $P = (v_1,\dots, v_{\ell+1})$ satisfying \ref{lem:existence P_0:vertices}--\ref{lem:existence P_0:matched edges} such that $C_0 = C_0(M)$ is a cycle in $M_1\cup M_2$ containing the two unsaturated vertices $u_1$ and $u_2$. In particular, property~\ref{lem:existence P_0:matched edges} implies that \mbox{$M\cap P_0 \subseteq M_c$} for some~$c\in\{1,2\}$, and we write $\bar{c} = 3 - c$. Since $P_0$ is $M$-alternating, this implies that $P_0$ is {$M_c$-alternating} as well. Noting that $v_{\ell-1}\in C_0(M)\cup \cal{C}(M)$ whereas $v_\ell\notin C_0(M)\cup \cal{C}(M)$ and $v_\ell v_{\ell+1}\in M_c\cap M$ implies that $v_{\ell-1}v_\ell\in M_3\setminus M$.

Consider now the $(M_c,M)$-alternating path $P_c(M)$ between $u_2$ and $u_1$ in $C_0(M)$ and let $h\ge 1$ denote its number of vertices. Note that $P_c(M)$ contains only edges of $C_0(M)\subseteq M_1\cup M_2$ and both its first and its last edge are in $M_c$. Therefore, $P_c(M)\star P_0$ is $M_c$-alternating.

Let $C$ be the component containing $v_\ell v_{\ell+1}$ in $M_{\bar{c}}\cup M$. Recall that the graph $M_{\bar{c}}\cup M$ contains at most one (non-edge) path component and all other components are either edges of $M_{\bar{c}}\cap M$ or $(M_{\bar{c}}, M)$-alternating cycles and that $P_{\bar{c}}(M) \subseteq C_0(M)$ is the only candidate for a nontrivial path component. We know by~\ref{lem:existence P_0:vertices} that $v_\ell,v_{\ell+1}\notin C_0(M)$. As $v_\ell v_{\ell+1}\in M_c\cap M$, it follows that $C$ must have at least two edges and must therefore be an $(M_{\bar{c}}, M)$-alternating cycle. By \ref{lem:existence P_0:vertices}, we know that~$C\notin \mathcal{C}(M)$ and hence, $C$ contains an edge of~$M\cap M_3$. Moreover, this edge is not incident to the vertex $v_\ell$, since $v_{\ell-1}v_\ell\in M_3$. This in particular implies that the length of $C$ is at least four. Let~$u$ be the second neighbor of~$v_\ell$ in $C$ besides $v_{\ell+1}$, and let $P(v_{\ell+1},u)$ be the path from $v_{\ell+1}$ to $u$ on $C$ avoiding~$v_{\ell}$. Note that~$P(v_{\ell+1},u)$ is $(M_{\bar{c}}, M)$-alternating and contains all edges of $M\cap C$ besides~$v_\ell v_{\ell+1}$; in particular, it contains an edge of $M\cap M_3$. Furthermore, $P_0\star P(v_{\ell+1},u)$ is $M$-alternating and all of its vertices except for $u_1$ are $M$-saturated.

Define $P = P_c(M)\star P_0\star P(v_{\ell+1},u) = (u_2, \ldots, u_1=v_1, \ldots, v_{\ell+1}, \ldots, u)$; write $P = (w_1,\dots, w_k)$.  Since $P_c(M)\star P_0$ is $M_c$-alternating with its last edge $v_\ell v_{\ell+1}$ in $M_c$ and $P(v_{\ell+1},u)$ is $M_{\bar{c}}$-alternating with its first edge in $M_{\bar{c}}$, it follows that $P$ is $c$-good. Furthermore,  $P_c(M)$ and $P_0\star P(v_{\ell+1},u)$ are $M$-alternating with unsaturated vertices $u_1$ and $u_2$ and the length $h-1$ of $P_c(M)$ is odd. Therefore, $P$ is nearly-$M$-alternating with $w_1=u_2$, $w_h=u_1$, and $M=M(P; 1,h)$, proving~\ref{existence P-unsaturated}. Since $P_c(M)\subseteq C_0(M)$, which is an $(M_1, M_2)$-alternating cycle, it follows that $P_c(M) = P_{1,h}$ is itself $(M_1,M_2)$-alternating and in particular $f_3(P,M) = 0$, establishing~\ref{existence P-no M3 between unmatched}. Since $C$ is $(M_{\bar{c}}, M)$-alternating and $v_\ell v_{\ell+1}\in M$ (by~\ref{lem:existence P_0:last edge}), we know that $v_\ell u\in M_{\bar{c}}$; so we let $t\in [k]$ be the index such that $v_\ell=w_t$. Since $C$ is an even cycle of length at least four, it follows that $t$ is odd and $t\leq k-3$. As explained earlier, $v_{\ell-1}v_\ell \in M_3\setminus M$, yielding~\ref{existence P-cycle}. Property~\ref{existence P-Mc alternating} follows from the fact that $P_c(M)\star P_0$ is $M_c$-alternating. Finally, $P(v_{\ell+1},u)$ contains an edge of $M_3\cap M$ and hence so does $(w_t,\dots, w_k)$, from which~\ref{existence P-matched M3 edge} follows. 
\end{proof}

\subsection{Proof of the Switching Lemma}

We are now ready to complete the proof of the Switching Lemma.
\begin{proof}[Proof of \cref{lem:main}]
Suppose that $G$ does not contain an $(a_1+1,a_2,a_3-1)$-matching. Then, by \cref{lem:existence P}, there exists an $(a_1,a_2,a_3)$-matching $M\subseteq G$ with unsaturated vertices $u_1$ and~$u_2$ and a nearly-$M$-alternating $c$-good path $P\subseteq G$ satisfying~\ref{existence P-unsaturated}--\ref{existence P-matched M3 edge}. Let $\bar{c} = 3-c$.
Write $P = (w_1, \ldots, w_k)$, and note that \cref{obs:unsaturated-positions} implies that $k$ is even; further, let $h$ be such that $w_1=u_2$, $w_h=u_1$, and $M = M(P;1,h)$, as in~\ref{existence P-unsaturated}.  By~\ref{existence P-cycle}, we know that $w_tw_k\in M_{\bar{c}}$ for some odd index~$t$; let $C$ denote the cycle $(w_t,\dots, w_k,w_t)$. By~\ref{existence P-matched M3 edge}, $C$ is an $(M_{\bar{c}},M)$-alternating cycle containing an edge of $M_3\cap M$.

Now, we apply \cref{lem: can move_n} repeatedly to $M$ and~$P$ until this is no longer possible; let $M(P;i',r)$ be the final matching. Note that, for us to not be able to apply \cref{lem: can move_n} to this matching, we must have $r>k-2$, or $i'+1=r$ and $v_{i'}v_r\in M_3$. In the latter case, $f_3(P,M(P;i',r))=-1$, and thus, by \cref{obs:changes} and the fact that $f_3(P,M) = 0$, we have $a_3(M(P;i',r)) = a_3-1$. We also know that~$a_2(M(P;i',r))=a_2$, and therefore, $M(P;i',r)$ is an $(a_1+1,a_2,a_3-1)$-matching, contradicting our assumption. Hence, $r>k-2$ and, because $k$ and $r$ are even (by \cref{obs:unsaturated-positions}), it follows that $r=k$ and~$a_2(M(P;i',k)) = a_2$.
 Let~$i$ be maximal such that $a_2(M(P;i,k)) = a_2$, and write $M_{i,k} = M(P;i,k)$. If $a_3(M_{i,k}) < a_3$, then applying \cref{lem: intermediate value_n} to $M = M(P;1,h)$, $P$, and~$M_{i,k}$ yields the existence of an $(a_1+1,a_2,a_3-1)$-matching, which is again a contradiction. So, we may assume that $a_3(M_{i,k})\geq a_3$. 
We consider two cases, depending on whether or not $w_i$ is in~$C$.

\smallskip 

\noindent\underline{Case 1}: $w_i\in C$. Since $C\setminus M\subseteq M_{\bar c}$, we have  $\size{P_{i,k}\cap M_3\cap M_{i,k}} = 0$ and therefore, $f_3(P,M_{i,k}) \leq 0$. 
If~$f_3(P,M_{i,k}) < 0$, then by \cref{obs:changes} it follows that $a_3(M_{i,k})< a_3$, a contradiction. Thus~$f_3(P,M_{i,k}) = 0$ and hence $a_3(M_{i,k}) = a_3$ by \cref{obs:changes}; it follows that $M_{i,k}$ is an $(a_1,a_2,a_3)$-matching of~$G$ with unsaturated vertices $w_i,w_k\in V(C)$. We will now show that, in the context of \cref{lem:M3edge} for $M_{i,k}$, we have $C_0(M_{i,k}) = C$. Since $C$ is not an $(M_1, M_2)$-alternating cycle by \cref{lem:existence P}\ref{existence P-matched M3 edge}, we obtain a contradiction. Indeed, using again the fact that~$C$ is an~$(M_{\bar{c}},M)$-alternating cycle (by~\ref{existence P-matched M3 edge}) and $f_3(P,M_{i,k})=0$, we deduce that $P_{i,k}$ contains no edges of~$M_3$ and is therefore the $(M_c,M_{i,k})$-alternating path $P_c(M_{i,k})$ between the unsaturated vertices~$w_i$ and $w_k$. Again by~\ref{existence P-matched M3 edge} and the fact that $M$ and $M_{i,k}$ agree outside of $P_{i,k}$, it follows that the path~$(w_k, w_{t}, w_{t+1},\dots, w_i)$ is the $(M_{\bar{c}}, M_{i,k})$-alternating path $P_{\bar{c}}(M_{i,k})$. Therefore $C = C_0(M_{i,k})$, as claimed.

\smallskip

\noindent\underline{Case 2}:  $w_i\notin C$, that is, $i\leq t-1$. Note that, since $i$ and $t$ are both odd, we have $i\leq t-2$. Hence~$w_t$ is saturated in $M_{i,k}$ by the $M_3$-edge $w_{t-1}w_t$.
First we construct a matching $M^\ast$ of size $n-1$ such that $a_2(M^\ast) = a_2(M_{i,k}) = a_2$ and $a_3(M^\ast) = a_3(M_{i,k}) -1\geq a_3-1$. 
If $c=2$, then $w_kw_t \in M_1$ and the matching $M^\ast :=M_{i,k} \setminus \{w_{t-1}w_{t}\} \cup \{w_kw_t\}$ has indeed one less $M_3$-edge and one more $M_1$-edge than $M_{i,k}$.

Suppose now that $c=1$. In this case $w_kw_t\in M_2\setminus M_{i,k}$. Since $w_i$ is $M_{i,k}$-unsaturated, $(w_1, \ldots , w_t)$ is $M_1$-alternating (by~\ref{existence P-Mc alternating}), and $i$ is odd, it follows that 
$w_iw_{i+1}\in M_1\setminus M_{i,k}$. The next edge, $w_{i+1}w_{i+2}$, is then contained in $M_{i,k}$, as well as $M_2$, as otherwise $M(P;i+2, k)$ would have the same number of $M_2$ edges as $M_{i,k}$, which would contradict the maximality of $i$. 
Then the matching $M^\ast = M_{i,k} \setminus \{w_{t-1}w_t, w_{i+1}w_{i+2}\} \cup \{w_kw_t, w_{i}w_{i+1}\}$ satisfies the desired properties.

Now, since in both cases
$a_3(M^\ast) = a_3(M_{i,k}) -1 \geq a_3 -1$ and $a_2(M^\ast) = a_2$, we can assume that $a_3(M^\ast) \geq a_3$, as otherwise $M^\ast$ represents a contradiction to the assumption that there is no $(a_1+1, a_2, a_3-1)$-matching. Consequently, while the matching $M^\ast$ has the correct number of $M_2$-edges, it still has more than the desired number of $M_3$-edges. To find an $(a_1+1,a_2,a_3-1)$-matching and obtain a final contradiction, we plan to use our Intermediate Value Lemma (\cref{lem: intermediate value_n}) again. For this, we first need another path $P'$ that is nearly-$M^\ast$-alternating: recall that $P$ does not contain the edge $w_kw_t$, which saturates the vertex $w_t$ in $M^{\ast}$, and is therefore not nearly-$M^\ast$-alternating. To remedy the situation, we simply ``cut'' $P$ and set $P' : = (w_1, \ldots , w_{t-1})$, which fits the bill.

To conclude the proof using \cref{lem: intermediate value_n}, we just need to construct another matching $\tilde{M}$ of size~$n-1$ such that $P'$ is nearly-$\tilde{M}$-alternating, $M^\ast \setminus P' = \tilde{M}\setminus P'$, $a_2(\tilde{M}) = a_2$, and $a_3(\tilde{M}) \leq a_3 -1$. 
The matching $M^{\Delta} := M\Delta C$ of size $n-1$ agrees with $M^\ast$ outside of $P'$ and therefore $M^\ast(P';i,j) = M^\Delta(P';i,j)$ for all suitable indices $i,j$; further, $P'$ is nearly-$M^\Delta$-alternating with unsaturated vertices $w_1$ and $w_h$, so $M^\Delta = M^\Delta(P';1,h)$.  
By shifting the smaller unsaturated vertex iteratively by two, one arrives at 
the matching $M^{\Delta}(P';h-1,h)$. All matchings during this procedure have the same number of $M_3$-edges, as the path $P_{1,h}$ between $w_1$ and $w_{h}$ is $(M_1,M_2)$- alternating (by property~\ref{existence P-no M3 between unmatched}). It also follows that the number of $M_2$-edges changes by exactly one in each step of this process. Hence, if we manage to show that $a_2(M^*)=a_2$ is between $a_2(M^\Delta)$ and~$a_2(M^\Delta(P';h-1,h))$, then we will also have shown that there exists a shift $M^\Delta(P';2h'+1,h) =: \tilde{M}$ whose number of $M_2$-edges is exactly $a_2$. Note also that, as $C$ contains an edge of $M_3\cap M$ by property~\ref{existence P-matched M3 edge}, we have~$a_3 (M^\Delta) < a_3$. Thus, $a_3(\tilde{M}) = a_3 (M^\Delta) < a_3 $, so we can apply the Intermediate Value Lemma to~$M^\ast$, $\tilde{M}$, and $P'$ to obtain our desired matching.  
We will derive the desired relation between the $a_2$-values using \cref{obs:changes}, which states that it is sufficient to derive the same inequalities for the $f_2$-function (in absolute value). In other words, we will show that
\begin{align}
\size{f_2(P', M^{\Delta}(P';1,h))}\geq \size{f_2(P', M^\ast)}\geq \size{f_2(P', M^{\Delta}(P';h-1,h)}.\label{f_2-interpolate}
\end{align}

We first claim that $\size{f_2(P', M)}  \geq \size{f_2(P', M^\ast)}$; since $f_2(P', M^{\Delta}(P';1,h)) = f_2(P', M(P';1,h)) = f_2(P', M)$, the first inequality in~\eqref{f_2-interpolate} follows immediately from this.
\cref{obs:changes}  applied to the matchings $M^{\Delta}$ and $M^\ast$ and the facts that $a_2(M^\ast) = a_2$ and $f_2(P', M^{\Delta}) = f_2(P', M)$ imply that
\begin{align*}
    a_2 &= a_2(M^\ast) = a_2(M^\Delta) - f_2(P', M^\Delta) + f_2(P', M^\ast)\\
    &= a_2 + (-1)^{c+1}\cdot \size{C \cap M_2} - f_2(P', M) + f_2(P', M^\ast),
\end{align*}
where the second line follows from the fact that, if $c=1$, then $a_2(M^\Delta) = a_2(M) + \size{C\cap M_2}$, while if $c=2$, then $a_2(M^\Delta) = a_2(M) - \size{C\cap M_2}$. Hence,
\begin{align*}
    f_2(P', M) = (-1)^{c+1}\cdot \size{C \cap M_2} + f_2(P', M^\ast). 
\end{align*}
Note that, since $P'$ is a good path, by \cref{obs:f-good-paths} the sign of $f_2(P', M^\ast)$ is the same as that of $f_2(P', M^{\Delta}) = f_2(P', M)$.
Observe further that $M\cap P_{1,h} \subseteq M_{\bar{c}}$, so $f_2(P', M) = (-1)^{c+1}\size{P_{1,h}\cap M_2}$, that is, all of $f_2(P', M), (-1)^{c+1}\size{C\cap M_2},$ and $f_2(P', M^\ast)$ have the same sign. Therefore
\begin{align*}
    \size{f_2(P', M)} =  \size{C \cap M_2} + \size{f_2(P', M^\ast)} \geq \size{f_2(P', M^\ast)}, 
\end{align*}
as claimed.

We now turn our attention to the second inequality in~\eqref{f_2-interpolate}. If $c=1$, the inequality is easy as~{$f_2(P', M^\Delta(P';h-1,h))=0$}. If on the other hand $c=2$, then $f_2(P', M^\ast)<0$, since $M^\ast$ leaves~$w_i$ and~$w_{t-1}$ unsaturated and $w_{t-2}w_{t-1}\in M_2$; in particular, this implies that $\size{f_2(P', M^\ast)}\geq 1$. Combined with the fact that $f_2(P', M^{\Delta}(P';h-1,h)) = 1$, this completes the verification of~\eqref{f_2-interpolate}.

\medskip
For the final step, recall that, by property~\ref{existence P-no M3 between unmatched},  for every integer $0\leq h' \leq h/2-1$, we have 
\begin{align*}
    \size{f_2(P', M^\Delta(P';2h'+1,h))} = \size{f_2(P', M^\Delta)} - h', 
\end{align*}
and that the signs of the evaluations of $f_2$ are the same since $P'$ is a good path. 
As a consequence, there exists some $0\leq h'\leq h/2-1$ such that $f_2(P', M^\Delta(P';2h'+1,h)) =  f_2(P', M^\ast)$. As argued earlier using \cref{obs:changes}, we have $a_2(M^\Delta(P';2h'+1,h)) = a_2(M^\ast) = a_2$ and a final application of \cref{lem: intermediate value_n} to $M^\ast, M^\Delta(P';2h'+1,h)$ and $ P'$ yields the existence of an $(a_1+1,a_2,a_3-1)$-matching. 
\end{proof}

\section{Concluding remarks}

In Theorem~\ref{thm:main} we required that the three matchings are disjoint. It is natural to wonder what happens if they are allowed to overlap. In fact, we know from~\cite{anastos2023splitting} that, when $a_1+a_2+a_3\leq n-2$, we can guarantee an $(a_1,a_2,a_3)$-matching even if the matchings are not disjoint (and the graph is not necessarily bipartite).
It turns out that the proof of our main theorem can be modified to include the case of multigraphs, i.e., when the three matchings $M_1$, $M_2$, and $M_3$ are not necessarily edge-disjoint. Hence, we obtain the following stronger result.  
\begin{theorem}\label{thm:main-multi}
    Let $G$ be a bipartite multigraph on $2n$ vertices which is the  union of three perfect matchings $M_1$, $M_2$, and $M_3$.
    Then, for any integers $a_1,a_2,a_3\in \mathbb{N}_0$ satisfying  $a_1 + a_2 + a_3 = n-1$,  the graph $G$ contains an $(a_1,a_2,a_3)$-matching.
\end{theorem}
As the proof of our main result is already fairly technical for simple graphs, we chose to present only the key modifications necessary to handle the multigraph case. 
The obvious, yet important, distinction between the two is whether the union of two matchings can contain a component which is a {single edge}. In the multigraph setting this will often be thought of as a (degenerate) cycle of length two. 
In the proof of \cref{lem:reduction-strengthening}, in the base case we need to consider the situation where $m'=\ell=1$. In this case though, since $2b_1'\geq m'$ by~\eqref{eq:b_1'b_2'}, it follows that we always have $b_1'\geq m'$ and we can proceed as before.
It is also important to observe that, if 
\cref{lem:M3edge} is not true and hence~$C_0(M)$ is a cycle in $M_1\cup M_2$, this cycle cannot be degenerate. Indeed, if the vertices $u_1$ and $u_2$ unsaturated by $M$ are connected by an $M_1$-edge, then we can simply add this edge to the matching and remove an arbitrary $M_3$-edge to obtain an $(a_1+1,a_2,a_3-1)$-matching. Similarly, the cycle $C$ in the proof of our Switching Lemma (specifically in \cref{lem:existence P}\ref{existence P-matched M3 edge} and \cref{lem:main}) cannot be degenerate. This is because $C\notin\mathcal{C}(M)$, implying that $C$ contains an $M_3$-edge, but $v_{\ell-1}v_{\ell}\in M_3$ and $M_3$ is a matching. Note that the cycles in $\mathcal{C}(M)$ may in fact be degenerate, but this does not affect the argument.

\subsection{Open problems}
It would be very interesting to extend our results to any $k\geq 4$ nonzero color-multiplicities.  
As the $k=n-1$ case would imply Montgomery's Theorem, resolving the problem in full generality is not expected to be easy --- unless of course a relatively simple direct reduction can be found. 
Tackling the  case $k=4$ fully already seems to require novel ideas.  

\begin{problem}
Let $G$ be a bipartite graph on $2n$ vertices whose edge set is decomposed into perfect matchings $M_1$, $M_2$, $M_3$, and $M_4$. Let $a_1, a_2, a_3, a_4 \in \mathbb{N}_0$ be  integers such that $a_1+a_2+a_3+a_4 = n-1$.  Then there exists a matching $M$ in $G$ such that $|M\cap M_i| =a_i$ for $i=1,2,3,$ and $4.$
\end{problem}
As it was already mentioned in \cite{anastos2023splitting}, it would also be worth investigating whether allowing the matchings to overlap changes the problem for larger $k$. In light of \cref{thm:main-multi} we tend to believe that the answer is no.

\medskip
A simpler task would be to show that the constant $1$ in \cref{con:AFMS} can be replaced by some function depending only on $k$. 

\begin{problem}
    Determine the smallest function $f:\mathbb{N} \to \mathbb{R}_{\geq 1}$ such that the following is true. Let $n$ be a (sufficiently large) integer, $G$ be a bipartite graph on $2n$ vertices that is the union of $k$ perfect matchings $M_1,\dots, M_k$, and $a_1,\dots, a_k\in  \mathbb{N}_0$ be integers summing up to $n-f(k)$. Then $G$ contains a matching $M$ such that $\size{M\cap M_i} = a_i$ for all $i\in [k]$. 
\end{problem}

\medskip
In another direction, it would be interesting to investigate whether \cref{thm:main} is true under weaker conditions. As discussed in~\cite[Remark~2]{anastos2023splitting}, if $n$ is even, $G$ is the disjoint union of $n/2$ copies of $K_4$, and $a_1,a_2,a_3$ are all odd and $a_1+a_2+a_3=n-1$, then $G$ has no $(a_1,a_2,a_3)$-matching. However, this is the only obstruction we are aware of. Taking $G$ to be bipartite automatically eliminates this example, but can we relax this assumption further? We believe that this should be the case and reiterate a conjecture of Anastos, Fabian, M\"uyesser, and Szab\'o~\cite[Conjecture~3]{anastos2023splitting}.

\begin{conjecture}[{\cite{anastos2023splitting}}]
    Let $G$ be a graph on $2n$ vertices which is the union of three disjoint perfect matchings, and suppose that $G$ has a component that is not isomorphic to $K_4$.
    Then, for any integers $a_1,a_2,a_3\in \mathbb{N}_0$ satisfying  $a_1 + a_2 + a_3 = n-1$, $G$ contains an $(a_1,a_2,a_3)$-matching.
\end{conjecture}

\section*{Acknowledgments} 
We would like to thank Clément Requilé for many helpful discussions.\\
\noindent(SB) Most of this research was conducted while the author was at the School of Mathematics, University of Birmingham, Birmingham, United Kingdom. The research leading to these results was supported by EPSRC, grant no.\ EP/V048287/1 and by ERC Advanced Grants ``GeoScape'', no.\ 882971 and ``ERMiD'', no.\ 101054936. There are no additional data beyond that contained within the main manuscript. \\
(MC) The research leading to these results was supported by the SNSF Ambizione Grant No. 216071.\\

\bibliographystyle{amsplain}
\bibliography{biblio}
\end{document}